\documentclass[11pt,a4paper,english,reqno]{amsart}
\usepackage{amsmath}
\usepackage{a4}
\usepackage{a4wide}
 \usepackage{amssymb, latexsym}
 \usepackage{bbm}
 \usepackage{mathtools}
 \usepackage{color}
 \usepackage{enumerate}

\theoremstyle{plain}
\newtheorem{thm}{Theorem}[section] 

\theoremstyle{plain}
\newtheorem{Prop}[thm]{Proposition}
\newtheorem{lemma}[thm]{Lemma}
\newtheorem{Cor}[thm]{Corollary}
\theoremstyle{definition}

\theoremstyle{remark}
\newtheorem{Rem}[thm]{Remark}
\newtheorem{Def}[thm]{Definition}

\DeclareMathOperator*{\osc}{osc}

\begin{document}
\title[A CLT for the Riemann zeta-function over a Boolean type transformation]
 {A Central limit theorem for the Birkhoff sum of the Riemann zeta-function over a Boolean type transformation}
\author[Schindler]{Tanja I. Schindler}
%
\keywords{central limit theorem, Riemann zeta-function, transfer operator, Boolean transformation}
 \subjclass[2010]{
    Primary:  11M06
    Secondary: 37A05, 37A30, 37A45, 60G10, 60F05}
 \date{\today}
\begin{abstract}
We prove a central limit theorem for the real and imaginary part and the absolute value of the Riemann zeta-function sampled along a vertical line in the critical strip with respect to an ergodic transformation similar to the Boolean transformation.
This result complements a result by Steuding who has proven a strong law of large numbers for the same system.
As a side result we state a general central limit theorem for a class of unbounded observables on the real line over the same ergodic transformation. The proof is based on the transfer operator method.
\end{abstract}
\maketitle

\section{Introduction and statement of main results}

In recent years there has been some interest to sample the Riemann zeta-function $\zeta\left(z\right)=\sum_{n=1}^{\infty}1/n^z$ along a vertical line $z=s+i\mathbb{R}$ for fixed $s\in(0,1)$. 
More precisely, for fixed $s$ we look on the limit of
\begin{align}
 \frac{\sum_{k=0}^{n-1}\zeta\left(s+i\cdot Z_k\right)}{n},\label{eq: def eq strong law}
\end{align}
where $Z_n$ is a sequence of random variables taking values in $\mathbb{R}$.

The first investigation was done by Lifshits and Weber in \cite{lifshits_sampling_2009}, 
where $s$ is fixed to be $1/2$ and $(Z_k)$ is set to be a Cauchy random walk, i.e.\
we define $(X_k)$ to be a sequence of independent, identically distributed random variables, obeying a Cauchy distribution, and $Z_k$ as the sum $Z_k\coloneqq \sum_{j=1}^k X_j$. 
For this system it was proven that almost surely \eqref{eq: def eq strong law} equals $1$.
These work was later generalized by Shirai, see \cite{shirai_variance_2008}, where $(Z_k)$ was supposed to be a symmetric $\alpha$-stable process with $\alpha\in [1,2]$.

Related to this work Steuding was considering $(Z_k)$ as the orbit of an ergodic transformation, see \cite{steuding_sampling_2012}.
To be more precise, we consider $\phi\colon \mathbb{R}\to\mathbb{R}$ being defined as
\begin{align}
 \phi(x)\coloneqq\begin{cases} \frac{1}{2}\cdot \left(x-\frac{1}{x}\right)&\text{if }x\neq 0\\
               0&\text{if }x=0
              \end{cases}\label{eq: def T}
\end{align}
and set $Z_k$ to be the $k-1$-th iterate of $\phi$, i.e.\ $Z_k(x)\coloneqq \phi^{k-1}(x)$.
Steuding proved that for almost all $x\in\mathbb{R}$ a finite limit of \eqref{eq: def eq strong law} exists also for this choice $(Z_k)$. 
The ergodic transformation $\phi$ is related to the classical Boolean transformation $\widetilde{\phi}\colon \mathbb{R}\to\mathbb{R}$ given by $\widetilde{\phi}\left(x\right)\coloneqq x-\frac{1}{x}$ if $x\neq 0$ and $\widetilde{\phi}\left(0\right)\coloneqq 0$. 
However, these two transformations are fundamentally different in the sense that $\phi$ is measure preserving and ergodic with respect to the probability measure $\mu$ with $\mu(x)=\mathrm{d}x/(\pi\cdot (1+x^2))$, for details see for example \cite{prykarpatski_ergodic_2015}, 
whereas $\widetilde{\phi}$ is measure preserving and ergodic with respect to the infinite, but $\sigma$-finite Lebesgue measure $\lambda$ on $\mathbb{R}$, see \cite{adler_ergodic_1973}.

From this point of view it makes sense to look at the transformation $\phi$ and not on $\widetilde{\phi}$ as the limit in \eqref{eq: def eq strong law} might not exist for $Z_k(x)\coloneqq {\widetilde{\phi}}^{k-1}(x)$ and almost all $x\in\mathbb{R}$. 
Particularly, the limit cannot exist if we look at $\sum_{k=1}^{n}\left|\zeta\left(s+i\cdot Z_k\right)\right|/n$ instead of the expression in \eqref{eq: def eq strong law}.
This follows immediately from \cite{aaronson_ergodic_1977}.

The results by Steuding have also been generalized, both in terms of observables replacing the Riemann zeta-function as well as in terms of transformations replacing $\phi$: Elaissaoui and Guennoun used $\left|\log\zeta\right|$ as the observable and a slight variation of $\phi$, see \cite{elaissaoui_logarithmic_2015}. Furthermore, Lee and Suriajaya studied a number of different kinds of meromorphic functions like Dirichlet $L$-functions or Dedekind-$\zeta$-functions and $\overline{\phi}$ an affine version of $\phi$, namely $\overline{\phi}(x)= \alpha/2\cdot \left((x+\beta)/\alpha-\alpha/(x-\beta)\right)$, for $\alpha>0$, $x\neq \beta$, see \cite{lee_ergodic_2017}. Finally, Maugmai and Srichan gave further generalizations of $\overline{\phi}$, see \cite{maugmai_mean-value_2019}. These transformations have been earlier studied in another context by Ishitani and Ishitani in \cite{ishitani_invariant_2007}.

The interest in theses kinds of strong laws arose from the possibility to state an equivalent of the Lindel\"of hypothesis but also to show that if one samples along $\phi$ on average the values of the Riemann zeta-function are small. 
To further quantify this behaviour we prove a central limit theorem showing that the distribution of the values (real part, imaginary part, and absolute value) behave in a nice way.

Central limit theorems for ergodic dynamical systems are a very classical object of study. Many of them make use of the transfer operator technique which we will also apply in our proof, see \cite{nagaev_limit_1957},
\cite{rousseau_theoreme_1982}, and \cite{guivarch_thoeremes_1988} for some of the earliest works. 
Also very particular central limit theorems with respect to a transformation from $\mathbb{R}$ into itself, similar to the transformation in \eqref{eq: def T} have been proven by Ishitani and Ishitani, see \cite[Theorem 4]{ishitani_invariant_2007} and Ishitani, see \cite[Theorem 2]{ishitani_transformations_2013}.
However, one of their requirements is that the observable is of bounded variation and thus has to be essentially bounded which implies that this theorem can not be applied on the Riemann zeta-function.

In the next section we will present our results first in a rather general setting as Theorem \ref{thm: main thm} and deduce the central limit theorem for the Riemann zeta-function as Corollary \ref{cor: CLT for zeta}.

\subsection{Statement of main results}
We first recall the definition of $\phi$ from \eqref{eq: def T}.  
Further, let $\lambda$ denote the Lebesgue measure on $\mathbb{R}$ and let the measure $\mu$ on $\mathbb{R}$ be defined by
\begin{align}
 \mathrm{d}\mu\left(x\right)\coloneqq \frac{1}{\pi\cdot\left(x^2+1\right)}\mathrm{d}\lambda(x).\label{eq: def measure}
\end{align}
For the following we will denote by $\mathbb{E}$ and $\mathbb{V}$ the standard definition of the expectation
and the variance with respect to $\mu$, i.e.\ we have for a function $\chi\colon \mathbb{R}\to\mathbb{R}$ 
that 
$\mathbb{E}\left(\chi\right)=\int \chi\mathrm{d}\mu$
and $\mathbb{V}\left(\chi\right)=\int \left(\chi-\mathbb{E}\left(\chi\right)\right)^2\mathrm{d}\mu$.
Furthermore, we denote by $h'\left(x-\right)$ and $h'\left(x+\right)$ the left and right derivative of a function $h$. If both derivatives exist and equal each other, we write $h'\left(x\right)$.

\begin{thm}\label{thm: main thm}
 Let $h\colon \mathbb{R}\to\mathbb{R}$ be such that the left and right derivatives exist and there exists $w\in(0,1/3)$ fulfilling
 \begin{align}
  h\left(x\right)&\ll \left|x\right|^{w}\qquad\text{ and }\qquad
  \max\left\{\left|h'\left(x-\right)\right|, \left|h'\left(x+\right)\right|\right\}\ll \left|x\right|^{w},\label{eq: boundedness condition}
 \end{align}
 as $x\to\pm\infty$. We set
\begin{align}
  S_n(\omega)\coloneqq \sum_{j=0}^{n-1}\left(h\circ \phi^j\right)(\omega).\label{eq: def Sn gen}
\end{align}
 
 Then we have that $\sigma^2\coloneqq \lim_{n\to\infty}\mathbb{V}\left(S_n\right)/n \in [0,\infty)$ and
 \begin{align*}
  \frac{S_n- \mathbb{E}\left(S_n\right)}{\sqrt{n}}\overset{D}{\longrightarrow}\mathcal{N}\left(0, \sigma^2\right),
 \end{align*}
 as $n\to\infty$ with respect to $\mu$, i.e.\ if $\sigma^2>0$ we have for all $a<b$ that
 \begin{align*}
  \lim_{n\to\infty}\mu\left(\omega\in\mathbb{R}\colon \frac{S_n(\omega)- \mathbb{E}\left(S_n\right)}{\sqrt{n}}\in (a,b)\right)=\frac{1}{\sqrt{2\pi}}\cdot \int_a^b\exp\left(-\frac{x^2}{2\sigma^2}\right)\mathrm{d}x.
 \end{align*}
 and if $\sigma^2=0$ we have the degenerate case that for all $\epsilon>0$
 \begin{align*}
  \lim_{n\to\infty}\mu\left(\omega\in\mathbb{R}\colon \left|\frac{S_n(\omega)- \mathbb{E}\left(S_n\right)}{\sqrt{n}}\right|>\epsilon\right)=0.
 \end{align*}
\end{thm} 

\begin{Rem}\label{rem: weak inv princ multi version}
 The proof of this theorem makes use of \cite[Theorem 3.1]{melbourne_statistical_2004} and as there are no additionally imposed conditions for the stronger statement of a weak invariance principle, see \cite[Theorem 3.4]{melbourne_statistical_2004}, the above result can also be generalized. Similarly, it can be extended to a multidimensional version.
\end{Rem}

\begin{Cor}\label{cor: CLT for zeta}
 Let $\zeta$ be the Riemann zeta-function.
 If $s\in(1/3,1)$ and $S_n$ is defined as one of the following Birkhoff sums 
 \begin{align}
  S_n(\omega)&\coloneqq \sum_{j=0}^{n-1}\Re\left(\zeta\left(s+i\cdot \phi^j(\omega)\right)\right),\text{ or }\label{eq: sum Re}\\
  S_n(\omega)&\coloneqq \sum_{j=0}^{n-1}\Im\left(\zeta\left(s+i\cdot \phi^j(\omega)\right)\right),\text{ or }\label{eq: sum Im}\\
  S_n(\omega)&\coloneqq \sum_{j=0}^{n-1}\left|\zeta\left(s+i\cdot \phi^j(\omega)\right)\right|,\label{eq: def Sn}
 \end{align}
 then $\sigma^2\coloneqq \lim_{n\to\infty}\mathbb{V}\left(S_n\right)/n\in [0,\infty)$ and it holds that
 \begin{align*}
  \frac{S_n- \mathbb{E}\left(S_n\right)}{\sqrt{n}}\overset{D}{\longrightarrow}\mathcal{N}\left(0, \sigma^2\right),
 \end{align*}
 as $n\to\infty$ with respect to $\mu$.
 
 If we set $s=1/2$, then we additionally have that $\sigma^2>0$ and thus the central limit theorem is non-degenerate.
\end{Cor}
\begin{Rem}
  In \cite[Theorem 1]{steuding_sampling_2012} some values for $\int\zeta\left(s+i\cdot x\right)\mathrm{d}\mu\left(x\right)$ are given explicitly from which one can immediately calculate $\mathbb{E}\left(S_n\right)$ for the first two definitions of $S_n$.

 However, for $\sigma^2$ it is hard to get a precise value or even an estimate.
 Considering \eqref{eq: sum Re} we might set $h\colon \mathbb{R}\to\mathbb{R}$ with $h\left(x\right)= \Re\left(\zeta\left(s+i\cdot x\right)\right)$ for fixed $s\in (1/3,1)$ and it follows easily that 
 \begin{align}
  \sigma^2=\mathbb{V}\left(h\right)
  +2\sum_{k=1}^{\infty}\mathrm{Cov}\left(h, h\circ \phi^{k}\right),\label{eq: variance in v and cov}
 \end{align}
 where $\mathrm{Cov}$ denotes the covariance with respect to $\mu$, i.e.\ we have for two functions $f,g\colon \mathbb{R}\to\mathbb{R}$ being square integrable with respect to $\mu$ that 
 $\mathrm{Cov}\left(f,g\right)=\int \left(f-\int f\mathrm{d}\mu\right)\cdot \left(g-\int g\mathrm{d}\mu\right)\mathrm{d}\mu$. 
 The first summand $\mathbb{V}\left(h\right)$ in \eqref{eq: variance in v and cov} is already difficult to calculate as $h$ is not an analytic function anymore (similarly if we are considering the imaginary or the absolute part instead of the real part). 
 The second sum term in \eqref{eq: variance in v and cov} is also difficult to estimate, $h$ and $h\circ \phi^{k}$ are clearly not independent, so the covariances do not boil down to zero. 
 However, following the proof of this theorem and applying \cite[Remark 3.7]{melbourne_statistical_2004} it will become clear that there is an exponential decay of correlations, i.e.\ there exist $R>0$ and $\theta\in (0,1)$ such that $\left|\mathrm{Cov}\left(h, h\circ \phi^{k}\right)\right|\leq R\cdot \theta^k\cdot \left\|\widetilde{h}\right\|_2\cdot\left\|\widetilde{h}\right\|$, where $\left\|\cdot\right\|$ denotes a norm which we will define in Section \ref{subsec: Banach space} and $\widetilde{h}$ a function related to $h$ fulfilling $\left\|\widetilde{h}\right\|<\infty$ being defined in Section \ref{subsec: main part of proof}.
 This of course implies that $\sigma^2<\infty$. 
 But finding the optimal constants $R$ and $\theta$ and estimating $\left\|\widetilde{h}\right\|$ is not completely immediate. 
 
 To insure that $\sigma^2\neq 0$, we have to make sure that $h$ is not a coboundary for $\phi$, i.e.\ that there does not exist any function $g$ such that $h=g-g\circ \phi$ holds almost surely. 
 There is a very strong numerical evidence that $\sigma^2>0$ does also hold in the case \eqref{eq: sum Im}. This will be discussed in Remark \ref{rem: numerics}.
\end{Rem} 

\begin{Rem}
 As discussed in Remark \ref{rem: weak inv princ multi version} it would also be possible to generalize this theorem as a multidimensional version having the real and imaginary part as its entries. However, in this case it is even more challenging to ensure that the covariance matrix is positive definite and thus the central limit theorem is non-degenerate. 
\end{Rem}

\begin{Rem}
 The bound $s\in (1/3, 1)$ looks strange at the beginning. However, it does not seem to be easy to soften that condition as it is directly associated to $w\in (1/3, 1)$ in \eqref{eq: boundedness condition}. 
 A further discussion about this issue will be given in Remark \ref{rem: bound of 1/3}.
\end{Rem}

\section{Proof of Theorems}
\subsection{Main part of the proof of Theorem \ref{thm: main thm}}\label{subsec: main part of proof}
In this section we give a skeleton of the proof of Theorem \ref{thm: main thm}. 
It will turn out that the main technical part to be shown is given as Proposition \ref{prop: main prop} below
and the proof of it will be given in Sections \ref{subsec: Banach space}, \ref{subsec: quasi-compactness}, and \ref{subsec: boundedness}.
In Section \ref{subsec: proof of Cor} we will give the proof of Corollary \ref{cor: CLT for zeta}.

\begin{proof}[Proof of Theorem \ref{thm: main thm}]
We define $I\coloneqq [0,1]$ and $\psi\colon I\to I$ by $\psi\left(x\right)\coloneqq 2x \mod 1$
and the function $\xi\colon I\to\mathbb{R}$ by $\xi\left(x\right)\coloneqq \cot\left(\pi x\right)$.
Note that $\xi$ is almost surely bijective. 

By \cite[Proposition 1.1]{prykarpatski_ergodic_2015} the measure $\mu$ is ergodic with respect to $\phi$.
If we denote by $\mathcal{B}_{\mathbb{R}}$ and $\mathcal{B}_{I}$ the Borel sets
on $\mathbb{R}$ and $I$ respectively and by $\lambda_I$ the Lebesgue measure restricted to $I$, then we have by
\cite[Proposition 1.2]{prykarpatski_ergodic_2015} that the dynamical systems 
$\left(\mathbb{R}, \mathcal{B}_{\mathbb{R}}, \mu, \phi\right)$ 
and 
$\left(I, \mathcal{B}_{I}, \lambda_I, \psi\right)$
are isomorphic via $\xi$, i.e.\
\begin{align*}
 \left(\phi\circ\xi\right)\left(x\right)=\left(\xi\circ\psi\right)\left(x\right),
\end{align*}
for all $x\in I$
and additionally $\xi$ and $\xi^{-1}$ are measure preserving, i.e.\
for all $B\in\mathcal{B}_{\mathbb{R}}$ it holds that $\mu\left(B\right)=\lambda_I\left(\xi^{-1} B\right)$
and for all $B\in\mathcal{B}_{I}$ it holds that 
$\lambda_I\left(B\right)=\mu\left(\xi B\right)$.

This gives us an easier system to study. 
Instead of studying the Birkhoff sum 
\begin{align}
 \sum_{n=0}^{N-1}\left(h\circ \phi^n\right)\left(x\right)\label{eq: second sum}
\end{align}
with $x\in \mathbb{R}$
we can study the sum 
\begin{align*}
 \sum_{n=0}^{N-1}\left(h\circ \xi\circ\psi^n\right)\left(y\right),
\end{align*}
for $y\in I$. 
Since the transformations $\phi$ and $\psi$ are isomorphic we can conclude that 
\begin{align*}
 \mu\left(\sum_{n=0}^{N-1}\left(h\circ \phi^n\right)\left(x\right)\in A\right)
=\lambda\left(\sum_{n=0}^{N-1}\left(h\circ \xi\circ\psi^n\right)\left(y\right)\in A\right),
\end{align*}
for all sets $A\in\mathcal{B}_{\mathbb{R}}$.

Formally, we define $\widetilde{h}\colon I\to\mathbb{R}$ by
$\widetilde{h}\left(x\right)\coloneqq \left(h\circ \xi\right)\left(x\right)$ 
and consider then the Birkhoff sum $\sum_{n=0}^{n-1} \widetilde{h}\circ \psi^n$. For this sum we will prove a central limit theorem using the transfer operator method. 

We first give the basic definition of the transfer operator. 
\begin{Def}
 If $\left(X,\mathcal{A},\nu, T\right)$ is a mixing, probability preserving dynamical system,
then we denote by $\widehat{T}$ the \emph{transfer operator} of $T$, i.e.\ the (up to almost sure equivalence)
uniquely defined operator such that for all $f\in\mathcal{L}^1_{\nu}\left(X\right)$ and $g\in\mathcal{L}^{\infty}_{\nu}\left(X\right)$ we have
\begin{align}
 \int f\cdot g\circ T\mathrm{d}\nu = \int \widehat{T} f\cdot g\mathrm{d}\nu.\label{eq: def transfer op}
\end{align}
\end{Def}
Furthermore, we will need the notion of quasi-compactness given as follows:
\begin{Def}\label{def: quasi-compact}
 $U\colon\mathcal{M}\to\mathcal{M}$ is a \emph{quasi-compact} operator if there exists a direct sum decomposition $\mathcal{M}=G\oplus H$ 
 and $0<\tau<\rho\left(U\right)$ with $\rho$ the spectral radius where
\begin{itemize}
\item $G$, $H$ are closed and $U$-invariant, i.e.\ $U\left(G\right) \subset G$, $U\left(H\right) \subset H$,
\item $\dim\left(G\right) <\infty$  and all eigenvalues of $U\lvert_{G}:\mathcal{F}\to \mathcal{F}$ have modulus larger than $\tau$, and
\item $\rho\left(U\lvert_{H}\right)<\tau$. 
\end{itemize}
\end{Def}
With this definition we are able to state the main proposition we need for our proof:
\begin{Prop}[{\cite[Theorem 3.1]{melbourne_statistical_2004}}]\label{prop: CLT dynsys}
 Let $\left(\Omega,\mathcal{B},\nu, T\right)$ be a probability measure space and let $T\colon \Omega\to\Omega$
 be an ergodic, measure preserving transformation. 
 Further let $\mathcal{M}\subset\mathcal{L}^2_{\nu}(\Omega)$ be a Banach space such that $\widehat{T}$ is quasi-compact on $\mathcal{M}$ and $\chi\in \mathcal{M}$.
 Then a central limit theorem for the sequence $\left(\chi\circ T^{n-1}\right)_{n\in\mathbb{N}}$ holds, i.e.\ 
if $S_n\coloneqq \sum_{k=0}^{n-1}\chi\circ T^k$, then we have for all $a<b$ that 
 \begin{align*}
  \lim_{n\to\infty}\nu\left(\omega\colon \frac{S_n(\omega)- \mathbb{E}\left(S_n\right)}{\sqrt{\mathbb{V}\left(S_n\right)}}\in (a,b)\right)=\frac{1}{\sqrt{2\pi}}\cdot \int_a^b\exp\left(-\frac{x^2}{2}\right)\;\mathrm{d}x.
 \end{align*}
\end{Prop}
Hence, we are left to show the following proposition:
\begin{Prop}\label{prop: main prop}
 There exists a Banach space $\left(\mathcal{M},\left\|\cdot\right\|\right)$ of functions mapping $I$ to $\mathbb{R}$ fulfilling the following conditions: 
\begin{enumerate}[(a)]
 \item\label{en: L2} $\mathcal{M}\subset\mathcal{L}^2_{\mu}$,
 \item\label{en: bounded} $\left\|\widetilde{h}\right\|<\infty$,
 \item\label{en: spec gap} $\widehat{\psi}$ is quasi-compact on $\mathcal{M}$.
\end{enumerate}
\end{Prop}
\eqref{en: L2} and \eqref{en: spec gap} guarantee that Proposition \ref{prop: CLT dynsys} is applicable 
for all functions in $\mathcal{M}$ and \eqref{en: bounded} implies that $\widetilde{h}\in\mathcal{M}$ and thus we can apply Proposition \ref{prop: CLT dynsys} on $\widetilde{h}$ giving the statement of the theorem. 
\end{proof}

The remaining part of the proof is structured as follows: In Section \ref{subsec: Banach space} we introduce a Banach space and prove basic properties for it. In Section \ref{subsec: quasi-compactness} we show that $\widehat{\psi}$ is a quasi-compact operator on this Banach space, i.e.\ \eqref{en: spec gap} is fulfilled and in Section \ref{subsec: boundedness} we show that $\widetilde{h}$ is bounded with respect to the Banach space norm $\left\|\cdot\right\|$, i.e.\ \eqref{en: bounded} holds. 
The statement in \eqref{en: L2} will turn out to be obvious from the construction of the Banach space, see Lemma \ref{lem: L2}.

\subsection{Definition of the Banach space and first properties}\label{subsec: Banach space}
For a measurable function $f\colon I\to\mathbb{R}$ and a Borel subset $S$ of $I$ we define the oscillation on $S$ by
\begin{align*}
 \osc\left(f,S\right)\coloneqq \sup_{x\in S} f(x)-  \inf_{x\in S} f(x)
\end{align*}
and we set $\osc(\emptyset)\coloneqq 0$. 
Further, we denote by $B_{\epsilon}(x)$ the $\epsilon$-ball around $x$ and 
let $R_{\alpha}$ be an operator on the real valued functions on $I$ being defined as 
$R_{\alpha}f(x)\coloneqq x^{\alpha}\cdot\left(1-x\right)^{\alpha}\cdot f(x)$.
Then we define
\begin{align*}
 \left|f\right|_{\alpha,\beta}
 \coloneqq \limsup_{\epsilon\searrow 0} \int \frac{\osc\left(R_{\alpha}f, B_{\epsilon}(x)\right)}{\epsilon^{\beta}}\mathrm{d}\lambda_I(x).
\end{align*}
 Let
 \begin{align*}
  \left\|\cdot\right\|_{\alpha,\beta}=\left\|\cdot\right\|_2+\left|\cdot\right|_{\alpha,\beta}
 \end{align*}
  and set
 \begin{align*}
  \mathsf{V}_{\alpha,\beta}\coloneqq \left\{f\colon I\to\mathbb{R}\colon \left\|f\right\|_{\alpha,\beta}<\infty\right\}. 
 \end{align*}
Furthermore, we set  
 \begin{align*}
  \left\|\cdot\right\|'_{\alpha,\beta}=\left\|\cdot\right\|_1+\left|\cdot\right|_{\alpha,\beta}
 \end{align*}
  and
 \begin{align*}
  \mathsf{V}'_{\alpha,\beta}\coloneqq \left\{f\colon I\to\mathbb{R}\colon \left\|f\right\|'_{\alpha,\beta}<\infty\right\}. 
 \end{align*} 
A similar Banach space was considered in \cite{keller_generalized_1985} and \cite{blank_discreteness_1997} and in \cite{kessebohmer_intermediately_2019} for subshifts of finite type, \cite{kessebohmer_intermediately_2019},
however not using the smoothing operator $R_{\alpha}$. 
 
The reason we define two different Banach spaces (Lemma \ref{lem: norms} and Lemma \ref{lem: completeness} will show that the spaces are indeed Banach spaces) is that we will show that $\widehat{\psi}$ is quasi-compact on $\big(\mathsf{V}'_{\alpha,\beta},\left\|\cdot\right\|'_{\alpha,\beta}\big)$ and then conclude that $\widehat{\psi}$ is also quasi-compact on $\big(\mathsf{V}_{\alpha,\beta},\left\|\cdot\right\|_{\alpha,\beta}\big)$.

We start with showing that $\left\|\cdot\right\|_{\alpha,\beta}$ and $\left\|\cdot\right\|'_{\alpha,\beta}$ are indeed norms. 
\begin{lemma}\label{lem: norms}
For all $\alpha\in(0,1)$, $\beta\in(0,1]$ we have that
 $\left\|\cdot\right\|_{\alpha,\beta}$ and $\left\|\cdot\right\|'_{\alpha,\beta}$ are norms. 
\end{lemma}
\begin{proof}
 We have for $f,g\in \mathsf{V}_{\alpha,\beta}$ that 
\begin{align*}
 \left|f+g\right|_{\alpha,\beta}
 &=\limsup_{\epsilon\searrow 0} \int \frac{\osc\left(R_{\alpha}\left(f+g\right), B_{\epsilon}(x)\right)}{\epsilon^{\beta}}\mathrm{d}\lambda_I(x)\\
 &=\limsup_{\epsilon\searrow 0} \int \frac{\osc\left(R_{\alpha}f+R_{\alpha}g, B_{\epsilon}(x)\right)}{\epsilon^{\beta}}\mathrm{d}\lambda_I(x)\\
 &\leq \limsup_{\epsilon\searrow 0} \int \frac{\osc\left(R_{\alpha}f, B_{\epsilon}(x)\right)}{\epsilon^{\beta}}\mathrm{d}\lambda_I(x)
 +\limsup_{\epsilon\searrow 0} \int \frac{\osc\left(R_{\alpha}g, B_{\epsilon}(x)\right)}{\epsilon^{\beta}}\mathrm{d}\lambda_I(x)\\
 &=\left|f\right|_{\alpha,\beta}+\left|g\right|_{\alpha,\beta}
\end{align*}
and thus
 \begin{align*}
  \left\|f+g\right\|_{\alpha,\beta}
  =\left\|f+g\right\|_2+\left|f+g\right|_{\alpha,\beta} 
  \leq \left\|f\right\|_2+\left\|g\right\|_2+\left|f\right|_{\alpha,\beta}+\left|g\right|_{\alpha,\beta}
  =\left\|f\right\|_{\alpha,\beta}+\left\|g\right\|_{\alpha,\beta}.
 \end{align*}
It is obviously true that $\left\|af\right\|_{\alpha,\beta}=a\left\|f\right\|_{\alpha,\beta}$, for any positive $a$
and since $\left\|\cdot\right\|_2$ is already a norm and $\left|f\right|_{\alpha,\beta}=0$ if $f=0$ almost surely, we
know that $\left\|f\right\|_{\alpha,\beta}=0$ if and only if $f=0$ almost surely.

The proof for $\left\|\cdot\right\|'_{\alpha,\beta}$ follows analogously.
\end{proof}

In order to verify that $\mathsf{V}_{\alpha,\beta}$ and $\mathsf{V}'_{\alpha,\beta}$ are Banach spaces we have to verify completeness.
\begin{lemma}\label{lem: completeness}
For $\alpha\in(0,1)$ and $\beta\in (0,1]$ we have that
 $\mathsf{V}_{\alpha,\beta}$ and $\mathsf{V}'_{\alpha,\beta}$ are complete. 
\end{lemma}
\begin{proof}
We first show that $\mathsf{V}_{\alpha,\beta}$ is complete by following the proof in \cite[Lemma 2.3.17]{blank_discreteness_1997}. 
Let $\left(f_n\right)$ be a Cauchy sequence with respect to $\left\|\cdot\right\|_{\alpha,\beta}$. 
Then, in particular $\left(f_n\right)$ is also a Cauchy sequence with respect to $\left\|\cdot\right\|_{2}$,
we set $f$ as its limit. 
So our next step is to prove that $f\in \mathsf{V}_{\alpha,\beta}$. 
Since $\left(f_n\right)$ is a Cauchy sequence with respect to $\left\|\cdot\right\|_{\alpha,\beta}$, 
for each $\delta>0$ we can choose $L>0$
such that $\left\|f_{k}-f_{\ell}\right\|_{\alpha,\beta}<\delta$ for all $k,\ell>L$. Then we have that
\begin{align*}
 \left\|f_{k}-f_{\ell}\right\|_{\alpha,\beta}
 =\left\|f_{k}-f_{\ell}\right\|_{2}+
 \limsup_{\epsilon\searrow0}\frac{\int\osc\left(R_{\alpha}\left(f_{k}-f_{\ell}\right),B\left(\epsilon,x\right)\right)\;\mathrm{d}\lambda_I\left(x\right)}{\epsilon^{\beta}}<\delta.
\end{align*}
By Fatou's lemma and the linearity of the operator $R_{\alpha}$ we have that the limit $\ell\to \infty$ of the sequence $f_{\ell}$
on the right hand side exists and thus also on the left hand side which implies
\begin{align*}
 \left\|f_{k}-f\right\|_{\alpha,\beta}
 =\left\|f_{k}-f\right\|_{2}+\limsup_{\epsilon\searrow0}\frac{\int\osc\left(R_{\alpha}\left(f_{k}-f\right),B\left(\epsilon,x\right)\right)\;\mathrm{d}\lambda_I\left(x\right)}{\epsilon^{\beta}}<\delta.
\end{align*}
Thus, $f\in\mathsf{V}_{\alpha,\beta}$ 
and $\left(f_{n}\right)$ converges to $f$ with respect to $\left\|\cdot\right\|_{\alpha,\beta}$ giving completeness.

The proof for $\big(\mathsf{V}'_{\alpha,\beta},\left\|\cdot\right\|'_{\alpha,\beta}\big)$ follows analogously. 
\end{proof} 

That \eqref{en: L2} holds for $\mathsf{V}_{\alpha,\beta}$ is obvious and we obtain the following lemma.
\begin{lemma}\label{lem: L2}
 We have that $\mathsf{V}_{\alpha,\beta}\subset \mathcal{L}^2_{\mu}$.
\end{lemma}

\subsection{Quasi-compactness}\label{subsec: quasi-compactness}
In this section we will show quasi-compactness of the operator $\widehat{\psi}$ on $\big(\mathsf{V}_{\alpha,\beta},\left\|\cdot\right\|_{\alpha,\beta}\big)$. However, we will start with proving quasi-compactness on $\big(\mathsf{V}'_{\alpha,\beta},\left\|\cdot\right\|'_{\alpha,\beta}\big)$.
\begin{lemma}\label{lem: quasi-compact}
If $0<\alpha<\beta\leq 1$, then $\widehat{\psi}$ is quasi-compact on  $\big(\mathsf{V}'_{\alpha,\beta},\left\|\cdot\right\|'_{\alpha,\beta}\big)$.
\end{lemma}
To prove this lemma we will use the following lemma by Hennion and Herv\'e giving sufficient conditions for quasi-compactness which is based on results by Ionescu-Tulcea and Marinescu, \cite{ionescu_theorie_1950}, 
and Doeblin and Fortet, \cite{doeblin_sur_1937}.

\begin{lemma}[{\cite[Theorem II.5]{hennion_limit_2001}}]\label{lem: hennion herve}
Suppose $\left(\mathcal{L} ,\left\|\cdot\right\|\right)$ is a Banach space and $U:\mathcal{L}\to\mathcal{L}$ is a bounded linear
operator with spectral radius $\rho(U)$ equal to $1$. Assume that there exists a semi-norm $\left|\cdot\right|'$ with
the following properties:
\begin{enumerate}[(i)]
\item\label{Hennion 1} $\left|\cdot\right|'$ is continuous on $\mathcal{L}$.
\item\label{Hennion 2} $U$ is bounded on $\mathcal{L}$ with respect to $\left|\cdot\right|'$, i.e.\ there exists $M > 0$ such that $\left|U f \right|'\leq  M \left|f\right|'$, for all $f \in\mathcal{L}$.
\item\label{Hennion 3} There exist constants $r\in(0,1)$, $R>0$, and $n_0\in\mathbb{N}$ such that 
\begin{align}
\left\|U^{n_0} f\right\|\leq r^{n_0}\cdot\left\|f\right\|+R\cdot \left|f\right|',\label{mar tul ios}
\end{align}
for all $f \in\mathcal{L}$.
\item\label{Hennion 4} 
$U\left\{f\in\mathcal{L}\colon\left|f\right|'<1\right\}$ is precompact on $\left(\mathcal{L},\left|\cdot\right|'\right)$, i.e.\ for each sequence $\left(f_n\right)_{n\in\mathbb{N}}$ with values in $\mathcal{L}$ fulfilling $\sup_{n\in\mathbb{N}} \left|f_n\right|'\leq 1$ there exists
a subsequence $\left(n_k\right)$ and $g \in \mathcal{L}$ such that  
\begin{align*}
\lim_{k\to\infty}\left|U f_{n_k}-g\right|'=0.
\end{align*}
\end{enumerate}

Then $U$ is quasi-compact.
\end{lemma}

\begin{proof}[Proof of Lemma \ref{lem: quasi-compact}]
To prove quasi-compactness we will use Lemma \ref{lem: hennion herve} with seminorm $\left|\cdot\right|'=\left\|\cdot\right\|_1$
using similar ideas as in \cite{keller_generalized_1985} and \cite{blank_discreteness_1997}.

First we have to determine the transfer operator $\widehat{\psi}$.
For piecewise expanding interval maps this is a very well-known result. 
Let $\mathcal{I}\coloneqq\left(I_n\right)_{n\in\mathbb{N}}$ be a countable family
of closed intervals with disjoint interiors and for any $I_n$ such that the set $I_n\cap \left(I\backslash \Omega'\right)$ consists exactly of the endpoints of $I_n$.
Furthermore, we assume that $T$ fulfills the following properties:
\begin{itemize}
 \item (Adler's condition) $ T_n\coloneqq T\lvert_{\mathring{I}_n}\in \mathcal{C}^2$ and
 $ T''/\left( T'\right)^2$ is bounded on $\Omega'$.
 \item (Finite image condition) $\#\left\{ T I_n\colon I_n\in\mathcal{I}\right\}<\infty$.
 \item (Uniform expansion)
 There exists $m>1$ such that  $\left| T_n'\right|\geq m$ for all $n\in\mathbb{N}$.
 \item $ T$ is topologically mixing. 
\end{itemize}
Then the transfer operator can be written as 
\begin{align*}
 \widehat{T}f=\sum_{T(y)=x}\frac{f(y)}{\left|T'(y)\right|}.
\end{align*}
This is a standard result, for details see for example \cite[Chapter 3]{baladi_positive_2000}.

Since $\psi'=2$ almost surely, we immediately obtain that
\begin{align}
 \widehat{\psi}f(x)=\frac{1}{2}\cdot\left(f\left(\frac{x}{2}\right)+f\left(\frac{x+1}{2}\right)\right).\label{eq: transfer operator calculated}
\end{align}
It follows immediately, that $\widehat{\psi}$ is a bounded linear operator. 
That any eigenvalue can not exceed $1$ follows already from the defining relation in \eqref{eq: def transfer op}. Furthermore, the constant functions are clearly contained in $\mathsf{V}_{\alpha,\beta}$ and in $\mathsf{V}_{\alpha,\beta}'$ and we have that $\widehat{\psi}\mathbbm{1}=\mathbbm{1}$ implying that the spectral radius of $\widehat{\psi}$ indeed equals $1$.

\emph{Proof of \eqref{Hennion 1}:}
This is obviously true. 

\emph{Proof of \eqref{Hennion 2}:}
\eqref{eq: transfer operator calculated} implies 
\begin{align}
\left\|\widehat{\psi} f\right\|_{1}
&=\int\left|\frac{1}{2}\cdot \left(f\left(\frac{x}{2}\right)+f\left(\frac{x+1}{2}\right)\right)\right|\mathrm{d}\lambda_I(x)\notag\\
&\leq \frac{1}{2}\cdot \int\left|f\left(\frac{x}{2}\right)\right|\mathrm{d}\lambda_I(x) +\frac{1}{2}\cdot \int\left|f\left(\frac{x+1}{2}\right)\right|\mathrm{d}\lambda_I(x)\notag\\
&=\int\left|f\left(x\right)\right|\mathrm{d}\lambda_I(x)=\left\|f\right\|_1,\label{eq: |w|1 estim}
\end{align}
i.e.\  $\widehat{\psi}$ is bounded on $\mathsf{V}'_{\alpha,\beta}$ with respect to $\left\|\cdot\right\|_{1}$.

\emph{Proof of \eqref{Hennion 3}:}
We have that 
\begin{align}
 \left\|\widehat{\psi}f\right\|'_{\alpha,\beta}=\left\|\widehat{\psi}f\right\|_1+\left|\widehat{\psi}f\right|_{\alpha,\beta}.\label{eq: norm first estim}
\end{align}

In order to estimate the second summand we set 
\begin{align*}
 f_1(x)\coloneqq \frac{1}{2}\cdot f\left(\frac{x}{2}\right)
 \qquad\qquad\text{ and }\qquad\qquad f_2(x)\coloneqq \frac{1}{2}\cdot f\left(\frac{x+1}{2}\right).
\end{align*}
This yields 
\begin{align}
 \left|\widehat{\psi}f\right|_{\alpha,\beta}
 &= \limsup_{\epsilon\searrow0} \int_{0}^{1} \frac{\osc\left(R_{\alpha}\left(\widehat{\psi}f\right), B_{\epsilon}(x)\right)}{\epsilon^{\beta}}\;\mathrm{d}\lambda_I(x)\notag\\
 &= \limsup_{\epsilon\searrow0} \int_{0}^{1} \frac{\osc\left(R_{\alpha}\left(f_1+f_2\right), B_{\epsilon}(x)\right)}{\epsilon^{\beta}}\;\mathrm{d}\lambda_I(x)\notag\\
 &\leq \limsup_{\epsilon\searrow0} \int_{0}^{1} \frac{\osc\left(R_{\alpha}f_1, B_{\epsilon}(x)\right)}{\epsilon^{\beta}}\;\mathrm{d}\lambda_I(x)
 +\limsup_{\epsilon\searrow0} \int_{0}^{1} \frac{\osc\left(R_{\alpha}f_2, B_{\epsilon}(x)\right)}{\epsilon^{\beta}}\;\mathrm{d}\lambda_I(x).\label{eq: norm hat psi}
\end{align}
We start with the estimation of the first summand of \eqref{eq: norm hat psi}.
If we substitute $y=x/2$ in the first summand of \eqref{eq: norm hat psi}, then we obtain
\begin{align*}
 R_{\alpha}f_1(x)
 &=\frac{1}{2}\cdot x^{\alpha}\cdot (1-x)^{\alpha}\cdot f\left(\frac{x}{2}\right)
 =\frac{1}{2}\cdot \left(2y\right)^{\alpha}\cdot (1-2y)^{\alpha}\cdot f(y)\\
 &= \frac{1}{2}\cdot 2^{\alpha}\cdot y^{\alpha}\cdot (1-y)^{\alpha}\cdot f(y)\cdot \left(\frac{1-2y}{1-y}\right)^{\alpha}
 =\frac{1}{2}\cdot 2^{\alpha}\cdot R_{\alpha}f(y)\cdot \left(\frac{1-2y}{1-y}\right)^{\alpha}.
\end{align*}
If we set $g(y)\coloneqq \left((1-2y)/(1-y)\right)^{\alpha}$, then we obtain
\begin{align*}
 \limsup_{\epsilon\searrow0} \int_{0}^{1} \frac{\osc\left(R_{\alpha}f_1, B_{\epsilon}(x)\right)}{\epsilon^{\beta}}\;\mathrm{d}\lambda_I(x)
 &= 2^{\alpha}\cdot \limsup_{\epsilon\searrow0} \int_{0}^{1/2} \frac{\osc\left(\left(R_{\alpha}f\right)\cdot g, B_{\epsilon/2}(y)\right)}{\epsilon^{\beta}}\;\mathrm{d}\lambda_I(y).
 \end{align*}

If we substitute $\epsilon'=\epsilon/2$, then we obtain
\begin{align}
\limsup_{\epsilon\searrow0} \int_{0}^{1} \frac{\osc\left(R_{\alpha}f_1, B_{\epsilon}(x)\right)}{\epsilon^{\beta}}\;\mathrm{d}\lambda_I(x)
 &\leq 2^{\alpha}\cdot \limsup_{\epsilon'\searrow0} \int_{0}^{1/2} \frac{\osc\left(\left(R_{\alpha} f\right)\cdot g, B_{\epsilon'}(y)\right)}{\left(2\cdot \epsilon'\right)^{\beta}}\;\mathrm{d}\lambda_I(y)\notag\\
 &=2^{\alpha-\beta}\cdot\limsup_{\epsilon\searrow0} \int_{0}^{1/2} \frac{\osc\left(\left(R_{\alpha} f\right)\cdot g, B_{\epsilon}(y)\right)}{\epsilon^{\beta}}\;\mathrm{d}\lambda_I(y).\label{eq: norm hat psi1}
\end{align}
In the next steps we compare $\osc\left(\left(R_{\alpha} f\right)\cdot g, B_{\epsilon}(y)\right)$ with 
$\osc\left(R_{\alpha} f, B_{\epsilon}(y)\right)$.
We define the functions $f^+\coloneqq \max\left\{f, 0\right\}$ and $f^-\coloneqq\max\left\{-f, 0\right\}$ and have that
\begin{align}
 \osc\left(\left(R_{\alpha}f\right)\cdot g, B_{\epsilon}(y)\right)\leq \sum_{r\in \{+,-\}} \osc\left(\left(R_{\alpha}f^r\right)\cdot g, B_{\epsilon}(y)\right).\label{eq osc f g}
\end{align}
Furthermore, we note that $g(y)\geq 0$ for all $y\in [0,1/2]$ and for $r\in \{+,-\}$
and a measurable set $D$ we have
\begin{align*}
 \MoveEqLeft\osc\left(\left(R_{\alpha}f^r\right)\cdot g, D\right)\\
 &\leq \sup_{x\in D}R_{\alpha}f^r\left(x\right)\cdot \sup_{x\in D}g\left(x\right)-\inf_{x\in D}R_{\alpha}f^r\left(x\right)\cdot \inf_{x\in D}g\left(x\right)\\
 &= \left(\inf_{x\in D}R_{\alpha}f^r\left(x\right)+\osc\left(R_{\alpha}f^r, D\right)\right)\cdot \left(\inf_{x\in D}g\left(x\right)+\osc\left(g, D\right)\right)\\
 &\qquad-\inf_{x\in D}R_{\alpha}f^r\left(x\right)\cdot \inf_{x\in D}g\left(x\right)\\
 &=\inf_{x\in D}R_{\alpha}f^r\left(x\right)\cdot \osc\left(g, D\right)+\inf_{x\in D}g\left(x\right)\cdot \osc\left(R_{\alpha}f^r, D\right)+\osc\left(R_{\alpha}f^r, D\right)\cdot\osc\left(g, D\right)\\
 &=\sup_{x\in D}R_{\alpha}f^r\left(x\right)\cdot \osc\left(g, D\right)+\inf_{x\in D}g\left(x\right)\cdot \osc\left(R_{\alpha}f^r, D\right).
\end{align*}
Since $\osc\left(R_{\alpha}f,D\right)=\osc\left(R_{\alpha}f^+,D\right)+\osc\left(R_{\alpha}f^-,D\right)$, we obtain from \eqref{eq osc f g}
\begin{align}
 \osc\left(\left(R_{\alpha}f\right)\cdot g\right)
 &\leq 2\left\|R_{\alpha}f\right\|_{\infty}\cdot \osc\left(g, D\right)+\inf_{x\in D}g\left(x\right)\cdot \osc\left(R_{\alpha}f, D\right)\notag\\
 &\leq 2\left\|R_{\alpha}f\right\|_{\infty}\cdot \osc\left(g, D\right)+\osc\left(R_{\alpha}f, D\right).\label{eq: osc Rf g}
\end{align}
To estimate the first summand we first note that $\left\|R_{\alpha}f\right\|_{\infty}<\infty$.
Otherwise we would have an $\epsilon>0$ and an interval $I$ such that
$\osc\left(R_{\alpha} f, B_{\epsilon}(y)\right)=\infty$ for all $y\in I$ 
which immediately implies $\left\|f\right\|_{\alpha,\beta}=\infty$.

Furthermore, we notice that if $D$ is an interval of diameter $r$, i.e.\ $D=[d,d+r]$ we have that
\begin{align}
 \osc\left(g, D\right)&=\osc\left(g, [d,d+r]\right)\leq \sup_{x\in [d,d+r]} \left|g'(x)\right|\cdot r
 \leq \sup_{x\in I} \left|g'(x)\right|\cdot r
 \ll r,\label{eq: osc g D r to infty}
\end{align}
for $r$ tending to zero. In particular we have that $\osc\left(g, B_{\epsilon}(x)\right)\ll \epsilon$ as $\epsilon\searrow 0$ for all $x\in I$.

Combining this with \eqref{eq: osc Rf g} and taking the limes superior in \eqref{eq: norm hat psi1} yields 
\begin{align}
 \limsup_{\epsilon\searrow0} \int_{0}^{1} \frac{\osc\left(R_{\alpha}f_1, B_{\epsilon}(x)\right)}{\epsilon^{\beta}}\;\mathrm{d}\lambda_I(x)
 &=2^{\alpha-\beta}\cdot\limsup_{\epsilon\searrow0} \int_{0}^{1/2} \frac{\osc\left(R_{\alpha} f, B_{\epsilon}(y)\right)}{\epsilon^{\beta}}\;\mathrm{d}\lambda_I(y).\label{eq: norm hat psi1a}
\end{align}

The estimation of the second summand in \eqref{eq: norm hat psi} follows more or less analogously and we will only give the main steps. 
A substitution of $y=(x+1)/2$ yields
\begin{align*}
 R_{\alpha}f_2(x)
 =\frac{1}{2}\cdot 2^{\alpha}\cdot R_{\alpha}f(y)\cdot \left(\frac{2y-1}{1-y}\right)^{\alpha}.
\end{align*}

If we set $\widetilde{g}(y)\coloneqq \left((2y-1)/(1-y)\right)^{\alpha}$ and substitute $\epsilon$ similarly as in \eqref{eq: norm hat psi1} we obtain
\begin{align}
 \limsup_{\epsilon\searrow0} \int_{0}^{1} \frac{\osc\left(R_{\alpha}f_2, B_{\epsilon}(x)\right)}{\epsilon^{\beta}}\;\mathrm{d}\lambda_I(x)
 &\leq 2^{\alpha-\beta}\cdot\limsup_{\epsilon\searrow0} \int_{1/2}^1 \frac{\osc\left(\left(R_{\alpha} f\right)\cdot \widetilde{g}, B_{\epsilon}(y)\right)}{\epsilon^{\beta}}\;\mathrm{d}\lambda_I(y).\label{eq: norm hat2 psi1}
\end{align}
As the calculations in \eqref{eq osc f g}, \eqref{eq: osc Rf g}, and \eqref{eq: osc g D r to infty} do not change if we replace $g$ by $\widetilde{g}$, 
we obtain 
\begin{align}
  \limsup_{\epsilon\searrow0} \int_{0}^{1} \frac{\osc\left(R_{\alpha}f_2, B_{\epsilon}(x)\right)}{\epsilon^{\beta}}\mathrm{d}\lambda_I(x)
   \leq 2^{\alpha-\beta}\cdot\limsup_{\epsilon\searrow0} \int_{1/2}^{1} \frac{\osc\left( R_{\alpha}f, B_{\epsilon}(y)\right)}{\epsilon^{\beta}}\mathrm{d}\lambda_I(y).\label{eq: norm hat psi2}
\end{align}

Combining then \eqref{eq: norm hat psi} with \eqref{eq: norm hat psi1a} and \eqref{eq: norm hat psi2} yields
\begin{align}
 \left|\widehat{\psi}f\right|_{\alpha,\beta}
 &\leq 2^{\alpha-\beta}\cdot\limsup_{\epsilon\searrow0} \int_{0}^{1} \frac{\osc\left(f, B_{\epsilon}(y)\right)}{\epsilon^{\beta}}\;\mathrm{d}\lambda_I(y)
 =2^{\alpha-\beta}\cdot\left|f\right|_{\alpha,\beta}.\label{eq: hennion herve applied}
\end{align}
Combining this with \eqref{eq: norm first estim} and \eqref{eq: |w|1 estim} yields
\begin{align*}
 \left\|\widehat{\psi}f\right\|'_{\alpha,\beta}
 \leq \left\|f\right\|_1+2^{\alpha-\beta}\cdot\left|f\right|_{\alpha,\beta}
 =2^{\alpha-\beta}\cdot\left\|f\right\|'_{\alpha,\beta}+\left(1-2^{\alpha-\beta}\right)\cdot \left\|f\right\|_1
\end{align*}
and since we assumed that $\alpha<\beta$ \eqref{Hennion 3} is fulfilled for $n_0=1$.

\emph{Proof of \eqref{Hennion 4}:}
We first prove that $\mathsf{K}\coloneqq\{f:\left\|f\right\|'_{\alpha,\beta}\leq 1\}$ is compact in $\mathcal{L}^1_{\mu}$
using the approach of \cite[Lemma 1.4, Lemma 1.7 and Theorem 1.13]{keller_generalized_1985}  or \cite[Lemma 2.3.18]{blank_discreteness_1997}.
Since $f\in\mathcal{L}^1_{\mu}$,
for each $\epsilon>0$ there exists $k\in\mathbb{N}$
such that
\begin{align}
 \int_0^{2^{-k}}\left|f\right|\mathrm{d}\lambda_I+\int_{1-2^{-k}}^1\left|f\right|\mathrm{d}\lambda_I<\epsilon.\label{eq: front and back integral}
\end{align}
We set $\delta_{\epsilon}\coloneqq 2^{-k}$, where $k$ is the smallest integer such that \eqref{eq: front and back integral} is fulfilled.

For the following we will associate the doubling map with the $2$-shift. Namely, we associate with each number in $I$ its binary expansion which is unique up to countably many points. With respect to this representation the doubling map acts as the shift transformation, i.e.\ if $x=x_1x_2\ldots$ in the binary expansion, then we have that $Tx=x_2x_3\ldots$. 

We denote by $\mathcal{B}_m$ the sigma algebra generated by the cylinder sets of length $m$.
By cylinder of length $m$ we mean sets of the form $\left\{x\in I\colon x=0.a_1\ldots a_m x_{m+1}x_{m+2}\ldots\right\}$, where 
$a_1\ldots a_m\in \{0,1\}^m$ is given and $x$ is represented in the binary system.
This implies that each atom in $\mathcal{B}_m$ has diameter $2^{-m}$. 

We choose $m\in\mathbb{N}$ such that $2^{-m}\leq \delta_{\epsilon}$ and 
\begin{align}
 m\geq -\cfrac{\log\left(\cfrac{\epsilon\cdot \delta_{\epsilon}^{\alpha}}{8}\right)}{\log\left(2\cdot \beta\right)}.\label{eq: choice m}
\end{align}
Furthermore, we note that the conditional expectations $\mathbb{E}\left(f\vert\mathcal{B}_m\right)$ 
are functions piecewise constant on the cylinder sets.
In particular, $m\geq k$ implies that 
$\int_0^{\delta_{\epsilon}} \mathbb{E}\left(f\vert\mathcal{B}_m\right)\mathrm{d}\lambda_I=\int_0^{\delta_{\epsilon}} f\mathrm{d}\lambda_I$
as well as $\int_{1-\delta_{\epsilon}}^1 \mathbb{E}\left(f\vert\mathcal{B}_m\right)\mathrm{d}\lambda_I=\int_{1-\delta_{\epsilon}}^1 f\mathrm{d}\lambda_I$.
This implies
\begin{align}
 \MoveEqLeft\int_{0}^{\delta_{\epsilon}}\left|f-\mathbb{E}\left(f\vert\mathcal{B}_m\right)\right|\mathrm{d}\lambda_I
 +\int_{1-\delta_{\epsilon}}^1\left|f-\mathbb{E}\left(f\vert\mathcal{B}_m\right)\right|\mathrm{d}\lambda_I\notag\\
 &\leq \int_{0}^{\delta_{\epsilon}}\left|f\right|\mathrm{d}\lambda_I+\int_0^{\delta_{\epsilon}}\left|\mathbb{E}\left(f\vert\mathcal{B}_m\right)\right|\mathrm{d}\lambda_I
 +\int_{1-\delta_{\epsilon}}^1\left|f\right|\mathrm{d}\lambda_I+\int_{1-\delta_{\epsilon}}^1\left|\mathbb{E}\left(f\vert\mathcal{B}_m\right)\right|\mathrm{d}\lambda_I
 \leq 2\epsilon.\label{eq: f-Gn f 0}
\end{align}

Let $\mathbb{A}_{m,\epsilon}$ be the set of atoms in $\mathcal{B}_m\cap [\delta_{\epsilon}, 1-\delta_{\epsilon}]$.
Then we have that
\begin{align*}
 \int_{\delta_{\epsilon}}^{1-\delta_{\epsilon}}f-\mathbb{E}\left(f\vert\mathcal{B}_m\right)\mathrm{d}\lambda
 &\leq \frac{1}{\inf_{x\in [\delta_{\epsilon}, 1-\delta_{\epsilon}]} x^{\alpha}\cdot \left(1-x\right)^{\alpha}}\cdot \int_{\delta_{\epsilon}}^{1-\delta_{\epsilon}}R_{\alpha}f-R_{\alpha}\mathbb{E}\left(f\vert\mathcal{B}_m\right)\mathrm{d}\lambda_I\notag\\
 &\leq 2\cdot \delta_{\epsilon}^{-\alpha}\cdot \int_{\delta_{\epsilon}}^{1-\delta_{\epsilon}}R_{\alpha}f-R_{\alpha}\mathbb{E}\left(f\vert\mathcal{B}_m\right)\mathrm{d}\lambda_I,
\end{align*} 
for $\epsilon$ and hence $\delta_{\epsilon}$ sufficiently small. 
Hence, 
\begin{align}
 \int_{\delta_{\epsilon}}^{1-\delta_{\epsilon}}f-\mathbb{E}\left(f\vert\mathcal{B}_m\right)\mathrm{d}\lambda_I
 &\leq 2\cdot \delta_{\epsilon}^{-\alpha}\cdot\sum_{\mathsf{A}\in\mathbb{A}_{m+1,\epsilon}}\osc\left(R_{\alpha}\left(f-\mathbb{E}\left(f\vert\mathcal{B}_m\right)\right),\left[\mathsf{A}\right]\right)\cdot \lambda\left(\left[\mathsf{A}\right]\right)\notag\\
 &\leq 2\cdot \delta_{\epsilon}^{-\alpha}\cdot\int_{\delta_{\epsilon}}^{1-\delta_{\epsilon}}\osc\left(\left(R_{\alpha}\left(f-\mathbb{E}\left(f\vert\mathcal{B}_m\right)\right)\right), B_{2^{-m}}\left(x\right)\right)\;\mathrm{d}\lambda_I\left(x\right).\label{eq: int f-Ef 2}
\end{align}

Furthermore, for $m$ sufficiently large we have that 
\begin{align*}
 \MoveEqLeft\int_{\delta_{\epsilon}}^{1-\delta_{\epsilon}}\osc\left(R_{\alpha}\left(f-\mathbb{E}\left(f\vert\mathcal{B}_m\right)\right), B_{2^{-m}}\left(x\right)\right)\;\mathrm{d}\lambda_I\left(x\right)\\
 &\leq \int_{0}^{1}\osc\left(R_{\alpha}\left(f-\mathbb{E}\left(f\vert\mathcal{B}_m\right)\right), B_{2^{-m}}\left(x\right)\right)\;\mathrm{d}\lambda_I\left(x\right)\\
 &\leq 2\cdot \limsup_{\epsilon\searrow 0}\int\frac{\int_{0}^{1}\osc\left(R_{\alpha}\left(f-\mathbb{E}\left(f\vert\mathcal{B}_m\right)\right), B_{\epsilon}\left(x\right)\right)}{\epsilon^{\beta}}\mathrm{d}\lambda_I\left(x\right)\cdot 2^{-m\cdot \beta}\\
 &=2\cdot \left|f-\mathbb{E}\left(f\vert\mathcal{B}_m\right)\right|_{\alpha,\beta}\cdot2^{-m\cdot\beta},
\end{align*}
for $m$ sufficiently large.

Combining this with \eqref{eq: int f-Ef 2} yields 
\begin{align} 
\int_{\delta_{\epsilon}}^{1-\delta_{\epsilon}}f-\mathbb{E}\left(f\vert\mathcal{B}_m\right)\mathrm{d}\lambda_I
 &\leq 4\cdot \delta_{\epsilon}^{-\alpha}\cdot  \left|f-\mathbb{E}\left(f\vert\mathcal{B}_m\right)\right|_{\alpha,\beta}\cdot2^{-m\cdot\beta}\notag\\
 &\leq 4\cdot \delta_{\epsilon}^{-\alpha}\cdot \left(\left|f\right|_{\alpha,\beta}+\left|\mathbb{E}\left(f\vert\mathcal{B}_m\right)\right|_{\alpha,\beta}\right)\cdot 2^{-m\cdot\beta}\notag\\
 &\leq 4\cdot \delta_{\epsilon}^{-\alpha}\cdot 2^{-m\cdot\beta}\leq \epsilon.\label{eq: f-Gn f}
\end{align}
The second but last inequality follows from the fact that $\mathbb{E}\left(f\vert\mathcal{B}_m\right)$
is a piecewise constant function with only finitely many jumps implying that $\left|\mathbb{E}\left(f\vert\mathcal{B}_m\right)\right|_{\alpha,\beta}=0$.
The last inequality follows from the choice of $m$ in \eqref{eq: choice m}.

In the following we fix an arbitrary sequence $\left(f_n\right)\subset \mathsf{K}$ and a new sequence of functions $(f_n^{\left(m\right)})_{n\in\mathbb{N}}\coloneqq\left(\mathbb{E}\left(f_n\vert\mathcal{B}_m\right)\right)_{n\in\mathbb{N}}$. 
For given $m\in\mathbb{N}$ we know that $(f_n^{\left(m\right)})$ is a sequence of bounded functions being piecewise constant on the same finite number of intervals.
Hence, there exists a subsequence $n\left(j,m\right)$ such that $(f_{n\left(j,m\right)}^{\left(m\right)})_{j\in\mathbb{N}}$ 
is a Cauchy sequence in $\mathcal{L}^1_{\mu}$ and thus converges. 
Additionally we might require the function $n:\mathbb{N}^2\to\mathbb{N}$ 
to be such that for each $m\in\mathbb{N}$ we have $\left\{n\left(j,m+1\right)\colon j\geq 1\right\}\subset\left\{n\left(j,m\right)\colon j\geq 1\right\}$. 
If we set $\overline{n}\left(j\right)\coloneqq n\left(j,j\right)$, then for each $m\in\mathbb{N}$ 
the sequence $(f_{\overline{n}\left(j\right)}^{\left(m\right)})_{j\in\mathbb{N}}$ is a Cauchy sequence in $\mathcal{L}^1_{\mu}$. 
We can conclude that for all $\epsilon>0$ and $m_0\in\mathbb{N}$ there exists $J\in\mathbb{N}$ such that for all $m\geq m_0$ and all $j,\ell\geq \max\left\{m_0,J\right\}$ we have that 
\begin{align}
 \left\|f_{\overline{n}\left(j\right)}^{\left(m\right)}-f_{\overline{n}\left(\ell\right)}^{\left(m\right)}\right\|_1<\epsilon.\label{eq: fnjk-fnlk}
\end{align}

In the last steps we combine \eqref{eq: f-Gn f 0}, \eqref{eq: f-Gn f}, and \eqref{eq: fnjk-fnlk} to obtain
\begin{align*}
\left\|f_{\overline{n}\left(j\right)}-f_{\overline{n}\left(\ell\right)}\right\|_1
&\leq\left\|f_{\overline{n}\left(j\right)}-f_{\overline{n}\left(j\right)}^{\left(m\right)}\right\|_1+\left\|f_{\overline{n}\left(j\right)}^{\left(m\right)}-f_{\overline{n}\left(\ell\right)}^{\left(m\right)}\right\|_1+\left\|f_{\overline{n}\left(\ell\right)}^{\left(m\right)}-f_{\overline{n}\left(\ell\right)}\right\|_1
\leq5\epsilon,
\end{align*}
if $j,\ell\geq \max\left\{m_0,J\right\}$
which proves that $\left(f_{\overline{n}\left(j\right)}\right)$ is a Cauchy sequence and thus convergent in $\mathcal{L}^1_{\mu}$. 
Hence, each sequence has a convergent subsequence by the completeness of $\mathsf{V}_{\alpha,\beta}$.
Arguing as in \cite[Lemma 1.12]{keller_generalized_1985} using $R_{\alpha}f_n$ instead of $h_n$ and 
$R_{\alpha}f$ instead of $h$ implies $\liminf_{n\to\infty}\left|f_n\right|_{\alpha,\beta}\leq \left|f\right|_{\alpha,\beta}$ and analogously as in \cite[Theorem 1.13]{keller_generalized_1985} we obtain 
\begin{align*}
 \left\|f\right\|'_{\alpha,\beta}
 &=\left\|f\right\|_1+\left|f\right|_{\alpha,\beta}
 \leq \liminf_{n\to\infty}\left\|f_n\right\|_1+\left|f\right|_{\alpha,\beta}
 =\left\|f\right\|_1+\left|f\right|_{\alpha,\beta}\leq 1.
\end{align*}
Since $\mathsf{K}$ is a compact subset of $\mathcal{L}^1_{\mu}$, its continuous image $\widehat{\psi}\mathsf{K}$ is compact in $\mathcal{L}^1_{\mu}$ as well and thus in particular precompact, i.e.\ \eqref{Hennion 4} holds.
\end{proof}

\begin{lemma}\label{lem: quasi-compact1}
If $0<\alpha<\beta\leq 1$, then $\widehat{\psi}$ is quasi-compact on  $\big(\mathsf{V}_{\alpha,\beta},\left\|\cdot\right\|_{\alpha,\beta}\big)$.
\end{lemma}
\begin{proof}
The main point of the proof is to use Lemma \ref{lem: quasi-compact} and making use of the fact that $\big(\mathsf{V}'_{\alpha,\beta},\left\|\cdot\right\|'_{\alpha,\beta}\big)$ is a larger space than $\big(\mathsf{V}_{\alpha,\beta},\left\|\cdot\right\|_{\alpha,\beta}\big)$.

Using the defining relation of $\widehat{\psi}$ in \eqref{eq: def transfer op} implies that its operator norm  with respect to the $\mathcal{L}_{\mu}^{1}$-norm is equal to $1$. This immediately implies that the modulus of any eigenvalue cannot exceed $1$. 
On the other hand, using the explicit form of the transfer operator from \eqref{eq: transfer operator calculated} we can calculate $\widehat{\psi}\mathbbm{1}=\mathbbm{1}$. 
We further note that the constant functions are contained in $\mathsf{V}_{\alpha,\beta}$ as well as in $\mathsf{V}'_{\alpha,\beta}$ implying that on both spaces the spectral radius is $1$.
By Lemma \ref{lem: quasi-compact} there exist subspaces $G,H$ of $\mathsf{V}'_{\alpha,\beta}$ fulfilling the properties as in Definition \ref{def: quasi-compact} and $\mathbbm{1}\in G$. 
Thus, there exist subspaces $\widetilde{G}\subset G$ and $\widetilde{H}\subset H$ such that $\mathbbm{1}\in\widetilde{G}$ and as subspaces of $G$ and $H$ they must also fulfill the properties of Definition \ref{def: quasi-compact}.
\end{proof}

\subsection{Boundedness of $\big\|\widetilde{h}\big\|_{\alpha,\beta}$}\label{subsec: boundedness}
In this section we prove Condition \eqref{en: bounded} making use of the conditions in \eqref{eq: boundedness condition}.

\begin{lemma}\label{lem: bounded}
There exist $0<\alpha<\beta<1$ such that 
$\big\|\widetilde{h}\big\|_{\alpha,\beta}<\infty$. 
\end{lemma}
\begin{proof}
We choose $\alpha$ and $\beta$ so that the following inequality is fulfilled:
\begin{align}
 w<\alpha<\beta<\frac{1}{3}.\label{eq: ineq s alpha beta}
\end{align}
Since we are assuming that $w<1/3$ in Theorem \ref{thm: main thm}, it is always possible to find such numbers $\alpha$ and $\beta$. 

 First we note that 
 \begin{align*}
  \lim_{x\to 0}\frac{\xi(x)}{x}=\frac{1}{\pi}\qquad\text{ and }\qquad\lim_{x\to 1}\frac{\xi(x)}{1-x}=\frac{1}{\pi}.
 \end{align*}
 This and Condition \eqref{eq: boundedness condition} imply 
 \begin{align}
  \left|\widetilde{h}(x)\right|\ll x^{-w}\cdot (1-x)^{-w},\label{eq: F cdot xdelta}
 \end{align}
 and also 
 \begin{align}
  \left|h'\left(\xi\left(x\right)\right)\right|\ll x^{-w}\cdot \left(1-x\right)^{-w}.\label{eq: F cdot xdelta1}
 \end{align}
 Here and in the following we understand the $\ll$-sign globally, i.e.\ by \eqref{eq: F cdot xdelta} we mean that there exists $K>0$ such that
 $\left|\widetilde{h}(x)\right|\leq K\cdot x^{-w}\cdot (1-x)^{-w}$, for all $x\in I$ and similarly for 
 \eqref{eq: F cdot xdelta1}. 
 We can conclude this since we were assuming that $h$ is continuous and the left and right derivatives exist.

 In order to estimate the modulus of the first derivative of $\widetilde{h}$ we notice that 
 \begin{align*}
  \left|\left.\frac{\mathrm{d}}{\mathrm{d}x}\,\widetilde{h}(x)\right|_{x=c}\right|
  &=\left|\left(\left.\frac{\mathrm{d}}{\mathrm{d}y}\, h\left(y\right)\right|_{y=\xi\left(c\right)}  \cdot \left.\frac{\mathrm{d}}{\mathrm{d}x}\, \xi(x)\right|_{x=c}\right)\right|
 \end{align*}
implying
\begin{align}
  \left|\frac{\mathrm{d}}{\mathrm{d}x}\, \widetilde{h}(x)\right|\ll\left|h'\left(\xi\left(x\right)\right)\right|\cdot \left|\xi'(x)\right|.\label{eq: deriv Fs}
 \end{align}
Furthermore, we have that 
 \begin{align}
  \xi'(x)=-\frac{\pi}{\sin^2\left(\pi\cdot x\right)}.\label{eq: deriv xi}
 \end{align}
If we set $g(x)\coloneqq x^{\alpha}\cdot\left(1-x\right)^{\alpha}\cdot \widetilde{h}(x)$, then
 \begin{align*}
  g'(x)= \alpha\cdot x^{\alpha-1}\cdot \left(1-x\right)^{\alpha}\cdot \widetilde{h}(x)-\alpha\cdot x^{\alpha}\cdot\left(1-x\right)^{\alpha-1}\cdot  \widetilde{h}(x)+  x^{\alpha}\cdot \left(1-x\right)^{\alpha}\cdot \widetilde{h}'(x).
 \end{align*}
We can conclude from  our choice in \eqref{eq: ineq s alpha beta} and from \eqref{eq: F cdot xdelta}
that 
 \begin{align*}
  \alpha\cdot x^{\alpha-1}\cdot \left(1-x\right)^{\alpha}\cdot \widetilde{h}(x)\ll x^{-1}\qquad\text{ and }\qquad
  \alpha\cdot x^{\alpha}\cdot \left(1-x\right)^{\alpha-1}\cdot \widetilde{h}(x)\ll (1-x)^{-1}.
 \end{align*}
On the other hand, \eqref{eq: deriv Fs} together with \eqref{eq: F cdot xdelta1} and \eqref{eq: deriv xi} yields
$x^{\alpha}\cdot\left(1-x\right)^{\alpha}\cdot \widetilde{h}'(x)\ll x^{-2}\cdot\left(1-x\right)^{-2}$
and thus $g'(x)\ll x^{-2}\cdot\left(1-x\right)^{-2}$.
This implies that there exist $K, \widetilde{K}$ such that for all $\epsilon>0$ and $x\in (\epsilon,1-\epsilon)$ it holds that
 \begin{align*}
  \osc\left(g, B_{\epsilon}(x)\right)
  \leq K\cdot \epsilon \cdot \sup_{y\in B_{\epsilon}(x)} g'(y)
  \leq \widetilde{K}\cdot \epsilon \cdot \left(x-\epsilon\right)^{-2}\cdot \left(1-x+\epsilon\right)^{-2}.
 \end{align*}
For the following we set 
\begin{align}
 K_{\epsilon}\coloneqq \epsilon^{\frac{1-\beta}{2}}\cdot \widetilde{K}^{1/2}+\epsilon.\label{eq: def Kepsilon}
\end{align}
We can choose $\epsilon$ sufficiently small such that for all $x\in[ K_{\epsilon},1-K_{\epsilon}]$ we have that
\begin{align}
 \osc\left(g, B_{\epsilon}(x)\right)
 &\leq \widetilde{K}\cdot \epsilon \cdot \left(K_{\epsilon}-\epsilon\right)^{-2}
 \cdot \left(1-K_{\epsilon}+\epsilon\right)^{-2}\notag\\
 &= \widetilde{K}\cdot \epsilon \cdot \left(\epsilon^{\frac{1-\beta}{2}}\cdot \widetilde{K}^{1/2}\right)^{-2}
 \cdot \left(1-\epsilon^{\frac{1-\beta}{2}}\cdot \widetilde{K}^{1/2}\right)^{-2}
 \leq 2\cdot \epsilon^{\beta}.\label{eq: osc g Beps}
\end{align}
Indeed, $K_{\epsilon}$ in \eqref{eq: def Kepsilon} is chosen such that 
$\osc\left(g, B_{\epsilon}(x)\right)\ll \epsilon^{\beta}$ as $\epsilon\searrow 0$. 

In order to consider the cases $x\leq K_\epsilon$ and $x\geq 1- K_{\epsilon}$ 
we notice that a similar argument as above using \eqref{eq: ineq s alpha beta} and 
\eqref{eq: F cdot xdelta} yields
\begin{align}
 \osc\left(g, B_{\epsilon}(x)\right)\ll 1.\label{eq: osc g Beps1}
\end{align}

In the next steps we estimate
$\int_0^{1}\osc\left(R_{\alpha}\widetilde{h}, B_{\epsilon}(x)\right)/\epsilon^{\beta}\mathrm{d}\lambda_I(x)$,
using \eqref{eq: osc g Beps} and \eqref{eq: osc g Beps1}.
We split the integral into the following parts:
\begin{align}
 \MoveEqLeft\int_0^{1} \osc\left(R_{\alpha}\widetilde{h}, B_{\epsilon}(x)\right)\mathrm{d}\lambda_I(x)\notag\\
 &= \int_0^{K_{\epsilon}}\osc\left(R_{\alpha}\widetilde{h}, B_{\epsilon}(x)\right)\mathrm{d}\lambda_I(x)
 + \int_{K_{\epsilon}}^{1-K_{\epsilon}}\osc\left(R_{\alpha}\widetilde{h}, B_{\epsilon}(x)\right)\mathrm{d}\lambda_I(x) \notag\\
 &\qquad+\int_{1-K_{\epsilon}}^{1}\osc\left(R_{\alpha}\widetilde{h}, B_{\epsilon}(x)\right)\mathrm{d}\lambda_I(x)\label{eq: int sep}
\end{align}
and by \eqref{eq: osc g Beps1} there exists $M>0$ such that we have for the first and last summand of \eqref{eq: int sep} that
\begin{align}
 \int_0^{K_{\epsilon}}\osc\left(R_{\alpha}\widetilde{h}, B_{\epsilon}(x)\right)\mathrm{d}\lambda_I(x)
 +\int_{K_{\epsilon}}^{1-K_{\epsilon}}\osc\left(R_{\alpha}\widetilde{h}, B_{\epsilon}(x)\right)\mathrm{d}\lambda_I(x)
 &\leq K_{\epsilon}\cdot M.\label{eq: int sep1}
\end{align}
In order to estimate the second summand of \eqref{eq: int sep} we notice that by \eqref{eq: osc g Beps} we have that
\begin{align}
 \int_{K_{\epsilon}}^{1-K_{\epsilon}}\osc\left(R_{\alpha}\widetilde{h}, B_{\epsilon}(x)\right)\mathrm{d}\lambda_I(x)
 &\leq 2\cdot \epsilon^{\beta},\label{eq: int sep2}
\end{align}
if $\epsilon$ is sufficiently small. 
If we combine \eqref{eq: int sep}, \eqref{eq: int sep1}, and \eqref{eq: int sep2}
we obtain 
\begin{align*}
 \int_0^{1} \osc\left(R_{\alpha}\widetilde{h}, B_{\epsilon}(x)\right)\mathrm{d}\lambda_I(x)
 \leq K_{\epsilon}\cdot M+2\cdot \epsilon^{\beta}.
\end{align*}
Hence, using the definition of $K_{\epsilon}$ in \eqref{eq: def Kepsilon}
and the condition on $\alpha$ and $\beta$ in \eqref{eq: ineq s alpha beta} yields 
\begin{align}
     \left|\widetilde{h}\right|_{\alpha, \beta}
 &\leq \limsup_{\epsilon \searrow 0}  K_{\epsilon}\cdot \epsilon^{-\beta}\cdot M+2
 =\limsup_{\epsilon \searrow 0}  \left(\epsilon^{\frac{1-\beta}{2}}\cdot \widetilde{K}^{1/2}+\epsilon\right)\cdot \epsilon^{-\beta}\cdot M+2\notag\\
 &\leq \limsup_{\epsilon \searrow 0}  \left(\widetilde{K}\cdot \epsilon^{\frac{1}{2}-\frac{3}{2}\cdot\beta}+\epsilon^{1-\beta}\right)\cdot M+2\label{eq: |h|alphabeta}
 <\infty.
\end{align}

Finally, we have to prove that $\widetilde{h}$ is an $\mathcal{L}^2_{\mu}$-function.
We have that 
\begin{align*}
 \int \widetilde{h}^2\mathrm{d}\lambda_I=\int h^2\circ \xi\mathrm{d}\lambda_I=\int h^2\mathrm{d}\mu.
\end{align*}
From \eqref{eq: boundedness condition} together with the differentiability that there exists $K>0$ such that for all $x\in\mathbb{R}$ we have that $h^2\left(x\right)\leq K\cdot \left|x\right|^{2w}$. 
Using \eqref{eq: def measure} and the fact that $w<1/2$ yields 
\begin{align*}
 \int h^2\left(x\right)\mathrm{d}\mu\left(x\right)
 &\leq \frac{K}{\pi}\cdot \int \frac{\left|x\right|^{2w}}{\left(x^2+1\right)}\;\mathrm{d}\lambda_I\left(x\right)<\infty.
\end{align*}
This proves \eqref{en: bounded}. 
\end{proof}

\begin{Rem}\label{rem: bound of 1/3}
 Looking closer at the proof we see that we indeed require $w<1/3$.
 We have to estimate the summands in \eqref{eq: int sep} separately. 
 $K_{\epsilon}$ in \eqref{eq: def Kepsilon} is chosen in a way to ensure $\int_{K_{\epsilon}}^{1-K_{\epsilon}}\osc\left(R_{\alpha}\widetilde{h}, B_{\epsilon}(x)\right)\mathrm{d}\lambda_I(x)\ll \epsilon^{\beta}$, as $\epsilon\searrow 0$, see \eqref{eq: int sep2}. This implies 
 \begin{align*}
 \limsup_{\epsilon\searrow 0}\frac{\int_{K_{\epsilon}}^{1-K_{\epsilon}}\osc\left(R_{\alpha}\widetilde{h}, B_{\epsilon}(x)\right)\mathrm{d}\lambda_I(x)}{\epsilon^{\beta}}<\infty
 \end{align*}
 and one can easily see that $K_{\epsilon}$ cannot be chosen smaller in magnitude. 
This in turn implies that in \eqref{eq: |h|alphabeta} we need for $\left|\widetilde{h}\right|_{\alpha,\beta}$ to be bounded that $1/2-3/2\cdot\beta>1$ giving the bound $\beta<1/3$. 
\end{Rem}

\subsection{Proof of Corollary \ref{cor: CLT for zeta}}\label{subsec: proof of Cor}
For proving Corollary \ref{cor: CLT for zeta} we will make use of the following lemma.
\begin{lemma}\label{lem: bound zeta}
 Let $s\in(0,1)$. Then for any $\delta>0$ and $k\in\mathbb{N}_0$ 
 we have that 
 \begin{align*}
  \zeta^{(k)}\left(s+i\cdot t\right)\ll_{k,\delta} \left|t\right|^{\frac{1-s}{2}+\delta},
 \end{align*}
 where $\zeta^{(k)}$ denotes the $k$-th derivative of the Riemann zeta-function $\zeta$ and $\zeta^{(0)}\coloneqq \zeta$. 
\end{lemma}
 This lemma can be deduced from \cite[Proof of Example 2.1]{lee_ergodic_2017}.
 For background see also \cite[p.\ 95/96]{titchmarsh_theory_1986}. 
\begin{proof}[Proof of Corollary \ref{cor: CLT for zeta}]
 In order to prove Corollary \ref{cor: CLT for zeta} we want to show that if we fix $s\in(1/3, 1)$ and set $h(x)=\Re\zeta\left(s+ix\right)$ (or $h(x)=\Im\zeta\left(s+ix\right)$ or $h(x)=\left|\zeta\left(s+ix\right)\right|$ respectively), then \eqref{eq: boundedness condition} is fulfilled. For this choice, clearly the sums \eqref{eq: def Sn gen} and \eqref{eq: sum Re} (or \eqref{eq: def Sn gen} and \eqref{eq: sum Im} or \eqref{eq: def Sn gen} and \eqref{eq: def Sn} respectively) coincide and Theorem \ref{thm: main thm} is applicable. 

 Given $s\in (1/3,1)$ we might set $\delta_s\coloneqq s/4-1/12$ implying that $w\coloneqq w_s\coloneqq (1-s)/2+\delta_s<1/3$.
 Hence, we can conclude from Lemma \ref{lem: bound zeta} that for each fixed $s\in (1/3,1)$ there exists $w<1/3$ such that 
 \begin{align*}
  \Re\zeta\left(s+i\cdot x\right)\ll \left|x\right|^w\quad\text{ and }\quad
  \Im\zeta\left(s+i\cdot x\right)\ll \left|x\right|^w\quad\text{ and }\quad 
  \left|\zeta\left(s+i\cdot x\right)\right|\ll \left|x\right|^{w},
 \end{align*}
 as $x\to\pm \infty$ giving the first statement in \eqref{eq: boundedness condition}. 
 
 Obviously, if $h(x)=\Re\zeta\left(s+ix\right)$, then its derivative can be written as 
 \begin{align*}
  h'(x)=\left.\frac{\partial}{\partial y}\, \Re\zeta\left(s+iy\right)\right|_{y=x}
 \end{align*}
 and analogously for the imaginary part.
 Since $\zeta$ is complex differentiable on $\mathbb{C}\backslash \left\{1\right\}$,
 in particular the partial derivatives $\partial/\partial y\, \Re\zeta\left(s+i y\right)$ and 
 $\partial/\partial y\, \Im\zeta\left(s+i y\right)$ exist for $s\in (1/3,1)$, see for example \cite[p.\ 32]{lang_complex_1999}.
 
 Furthermore, Lemma \ref{lem: bound zeta} applied for $k=1$ also implies that for each fixed $s\in (1/3,1)$ there exists $w<1/3$ such that 
 \begin{align}
  \zeta'\left(s+i\cdot x\right)\ll \left|x\right|^{w},\label{eq: zeta 1 bounded}
 \end{align}
 as $x\to\pm \infty$. 
 
 As $\zeta$ is complex differentiable if $s\in (1/3, 1)$ and $x\in \mathbb{R}$, the derivative with respect to $z=s+iy$ can be written as 
 \begin{align*}
  \frac{\mathrm{d}}{\mathrm{d} z}\,\zeta(z)=-i\cdot \frac{\partial }{\partial y}\, \Re\zeta\left(s+i\cdot y\right)+\frac{\partial }{\partial y}\, \Im\zeta\left(s+i\cdot y\right),
 \end{align*}
 using the Cauchy-Riemann formulas, see for example \cite[p. 32]{lang_complex_1999}. 
 But from this we obtain 
 \begin{align*}
  \left|\left.\frac{\partial }{\partial y}\, \Re\zeta\left(s+i\cdot y\right)\right|_{y=x}\right|
  &\leq \left(\left(\left.\frac{\partial }{\partial y}\, \Re\zeta\left(s+i\cdot y\right)\right|_{y=x}\right)^2+ \left(\left.\frac{\partial }{\partial x}\, \Im\zeta\left(s+i\cdot x\right)\right|_{y=x}\right)^2\right)^{1/2}\\
  &= \left|\zeta'\left(s+i\cdot x\right)\right|
 \end{align*}
 and similarly
 \begin{align*}
  \left|\left.\frac{\partial }{\partial y}\, \Im\zeta\left(s+i\cdot y\right)\right|_{y=x}\right|
  &\leq  \left|\zeta'\left(s+i\cdot x\right)\right|.
 \end{align*}

 Hence, for $s\in (1/3,1)$, \eqref{eq: zeta 1 bounded} implies that there exists $w<1/3$ such that
 \begin{align}
  \left.\frac{\partial }{\partial y}\,\Re\zeta\left(s+i\cdot y\right)\right|_{y=x}\ll \left|x\right|^{w}\quad\text{ and }\quad 
  \left.\frac{\partial }{\partial y}\, \Im\zeta\left(s+i\cdot y\right)\right|_{y=x}\ll \left|x\right|^{w},\label{eq: partial real and imag}
 \end{align}
 as $x\to\pm \infty$, which gives the second estimate in \eqref{eq: boundedness condition} and thus the first two statements of the corollary.
 
 Finally, we will prove the second estimate in \eqref{eq: boundedness condition} for \eqref{eq: def Sn}. 
 For those points, where $\zeta\left(s+iy\right)\neq 0$ we can calculate the derivative
  \begin{align*}
     \left|\frac{\partial}{\partial y}\, \left|\zeta\left(s+i y\right)\right|\right|
     &= \left|\frac{\partial}{\partial y}\, \sqrt{\left(\Re\zeta\left(s+i y\right)\right)^2+\left(\Im\zeta\left(s+i y\right)\right)^2}\right|\\
     &= \frac{1}{2}\cdot \frac{1}{\sqrt{\left(\Re\zeta\left(s+i y\right)\right)^2+\left(\Im\zeta\left(s+i y\right)\right)^2}}\\
     &\qquad\cdot\left|2\cdot \Re\zeta\left(s+i y\right)\cdot \frac{\partial }{\partial y}\, \Re\zeta\left(s+i y\right)+2\cdot \Im\zeta\left(s+i y\right)\cdot \frac{\partial}{\partial y}\, \Im\zeta\left(s+i y\right)\right|\\
     &\leq \max\left\{\left|\frac{\partial}{\partial y}\Re\zeta\left(s+i y\right)\right|, \left|\frac{\partial}{\partial y}\Im\zeta\left(s+i y\right)\right|\right\}
  \end{align*}
  and by taking limits we also have for those points on which $\zeta\left(s+iy\right)=0$ that the one-sided (left and right) partial derivatives denoted by $\partial/\partial y^-\, \left|\zeta\left(s+i y\right)\right|$ and $\partial/\partial y^+\, \left|\zeta\left(s+i y\right)\right|$ exist as well. 
  For these we also have 
  \begin{align*}
   \left|\frac{\partial}{\partial y^-}\, \left|\zeta\left(s+i y\right)\right|\right|\leq \max\left\{\left|\frac{\partial}{\partial y}\Re\zeta\left(s+i y\right)\right|, \left|\frac{\partial}{\partial y}\Im\zeta\left(s+i y\right)\right|\right\}
  \end{align*}
  and similarly for the right partial derivative. 
  Hence, we can conclude from \eqref{eq: partial real and imag} that 
  \begin{align*}
   \max\left\{\left.\frac{\partial }{\partial y^-}\,\left|\zeta\left(s+i\cdot y\right)\right|\right|_{y=x},\left.\frac{\partial }{\partial y^+}\,\left|\zeta\left(s+i\cdot y\right)\right|\right|_{y=x}\right\}\ll \left|x\right|^{w}
  \end{align*}
  and we can for all three cases \eqref{eq: sum Re}, \eqref{eq: sum Im}, and \eqref{eq: def Sn} conclude that a central limit theorem holds (possibly degenerate with $\sigma^2=0$).
  
We are left to show that $\sigma^2>0$. We use \cite[Proposition 3.2]{melbourne_statistical_2004} which 
around others states that it is enough to show that $h$ cannot be a coboundary with a transfer function from the quasi-compact Banach space, in our case $\mathsf{V}_{\alpha,\beta}$, i.e.\ there exists no 
function $f\in \mathsf{V}_{\alpha,\beta}$ such that $\widetilde{h}=f-f\circ \phi$ almost surely.
However, we can also choose a Banach space $\mathsf{V}_{\alpha,\beta}''$ which is a subspace of $\mathsf{V}_{\alpha,\beta}$ such that it is still quasi-compact. 
We set $\left\|\cdot\right\|_{\alpha,\beta}''\coloneqq \left\|R_{\alpha} \cdot\right\|+\left|\cdot\right|_{\alpha,\beta}$ and let $\mathsf{V}_{\alpha,\beta}''\coloneqq \left\{f\colon \mathbb{R}\to\mathbb{R}\colon f \text{ is continuous on }(0,1)\text{ and } \left\|f\right\|_{\alpha,\beta}''<\infty\right\}$. 
If we choose $\alpha<1/2$, then $\left\|f\right\|_{\alpha,\beta}''<\infty$ implies that $f\in\mathcal{L}^2$.
That $\mathsf{V}_{\alpha,\beta}''$ is a Banach space follows in the same way as the proof of Lemma \ref{lem: completeness} and the fact that $\mathsf{V}_{\alpha,\beta}''$ is still quasi-compact also follows in the same way as in Lemma \ref{lem: quasi-compact1}. 
If $\alpha$ is chosen in the same way as in Lemma \ref{lem: bounded}, then $\left\|\widetilde{h}\right\|_{\alpha,\beta}''<\infty$. 

Hence, we have to exclude that there exists a function $f\in\mathsf{V}_{\alpha,\beta}''$ such that 
$\widetilde{h}=f-f\circ \phi$ almost surely. 
$f\in\mathsf{V}_{\alpha,\beta}''$ implies in particular that $f$ is continuous. Thus
$\widetilde{h}=f-f\circ \phi$ almost surely is equivalent to $\widetilde{h}=f-f\circ \phi$ surely. 
Hence, $\widetilde{h}$ being a coboundary implies that for each periodic point $x$ of period $d$ that 
$\sum_{k=1}^d \big(\widetilde{h}\circ \phi^{k-1}\big)\left(x\right)=\sum_{k=1}^d\left(f\circ \phi^{k-1}-f\circ \phi^k\right)\left(x\right)=0$ which can easily be excluded numerically.
The points $1/3$ and $2/3$ are alternating periodic points with respect to $\phi$ as well as the points $1/7$, $2/7$, and $4/7$. 
Calculating the values with Maxima using double gives 
for the real part 
\begin{align*}
\widetilde{h}\left(1/3\right)+\widetilde{h}\left(2/3\right)
 &= \Re\zeta\left(1/2+i\cdot \cot\left(\pi/3\right)\right)+
\Re\zeta\left(1/2+i\cdot \cot\left(2\pi/3\right)\right)
\sim -0.632184187171495
\end{align*}
and for the absolute value
\begin{align*}
\widetilde{h}\left(1/3\right)+\widetilde{h}\left(2/3\right)= \left|\zeta\left(1/2+i\cdot \cot\left(\pi/3\right)\right)\right|+
\left|\zeta\left(1/2+i\cdot \cot\left(2\pi/3\right)\right)\right|\sim 1.99288350865465 
\end{align*}
each with $15$ valid digits.
As the imaginary part of $\zeta$ is an odd function, we cannot use the points $1/3$ and $2/3$, but we obtain
\begin{align*}
 \MoveEqLeft\widetilde{h}\left(1/7\right)+\widetilde{h}\left(2/7\right)+\widetilde{h}\left(4/7\right)\\
 &= \Im\zeta\left(1/2+i\cdot \cot\left(\pi/7\right)\right)+
\Im\zeta\left(1/2+i\cdot \cot\left(2\pi/7\right)\right)+
\Im\zeta\left(1/2+i\cdot \cot\left(4\pi/7\right)\right)\\
&\sim -0.448038121638635
\end{align*}
giving $\sigma^2>0$ if $s=1/2$.
\end{proof}
\begin{Rem}\label{rem: numerics}
 To also prove that $\sigma^2>0$ for other real values $s\neq 1/2$ one has to find periodic points $x$ with period $d$ such that $\sum_{k=1}^d \big(\widetilde{h}\circ \phi^{k-1}\big)\left(x\right)\neq 0$. This is numerically not difficult for single values. However, more rigorous work has to be done to prove that this is possible for all $s\in(1/3,1)$. 
\end{Rem}

\subsection*{Acknowledgements}
This work was partly done at the Research School of Finance, Actuarial Studies and Statistics at the Australian National University and partly at Scuola  Normale Superiore di Pisa.
The author acknowledges the support of the Centro di Ricerca Matematica Ennio De Giorgi and of UniCredit Bank R\&D group for financial support through the “Dynamics and Information TheoryInstitute”  at  the  Scuola  Normale Superiore.

The author would like to thank Ade Irma Suriajaya for her hospitality at Kyushu University and J\"orn Steuding for his hospitality at University of W\"urzburg. Furthermore, the author thanks the Women in Mathematics Special Interest Group of the Australian Mathematical Society for being awarded by a Cheryl E. Praeger Travel Award which facilitated the visit at Kyushu University. 

Finally, the author is grateful for useful discussions with Ade Irma Suriajaya, Junghun Lee and J\"orn Steuding.

\end{document}